\documentclass[12pt]{amsproc}
\usepackage{amssymb}
\usepackage{amsmath}
\usepackage{amsfonts}
\usepackage{geometry}
\usepackage{graphicx}
\providecommand{\U}[1]{\protect\rule{.1in}{.1in}}
\theoremstyle{plain}

\newtheorem{corollary}{Corollary}

\newtheorem{lemma}{Lemma}

\newtheorem{proposition}{Proposition}

\newtheorem{theorem}{Theorem}
\numberwithin{equation}{section}
\newcommand{\lap}{\mbox{$\triangle$}}

\begin{document}
\title[Indefinite fractional elliptic problem and Liouville theorems]{Indefinite fractional elliptic problem and Liouville theorems}
\author{Wenxiong Chen and Jiuyi Zhu}
\address{Wenxiong Chen
\\Department of Mathematics\\
Yeshiva University\\
New York, NY 10033, USA\\
Emails: wchen@yu.edu}
\address{ Jiuyi Zhu  \\
Department of Mathematics\\
Johns Hopkins University\\
Baltimore, MD 21218, USA\\
Emails:  jzhu43@math.jhu.edu  }
\thanks{\noindent }
\date{}
 \dedicatory{ }

\begin{abstract}
In this paper, we consider the indefinite fractional elliptic
problem. A corresponding Liouville-type theorem for the indefinite
fractional elliptic equations is established. Furthermore, we obtain
a priori bound for solutions in a bounded domain by blowing-up and
re-scaling. We also classify the solutions of some degenerate
elliptic equation originated from fractional Laplacian.
\end{abstract}

\maketitle
\section{Introduction}

The paper is to devote to studying the fractional Laplacian with
indefinite nonlinearity:
\begin{equation}
\left \{
\begin{array}{rll}
(-\lap)^{\frac{\alpha}{2}}u=&a(x)g(u) \quad \quad &\mbox{in} \
\Omega,\medskip \\
u>&0 \quad \quad \quad &\mbox{in} \ \Omega, \medskip \\
u=&0 \quad \quad \quad &\mbox{on} \ \partial \Omega, \medskip
\end{array}
\right. \label{frac}
\end{equation}
where $0<\alpha<2$, $\Omega$ is a smooth bounded domain in $\mathbb
R^n$ with $n>\alpha$. Concerning the function $a(x)$, we assume that $a(x)\in
C^2(\bar\Omega)$,
$$ \Omega^+:=\{x\in \Omega: a(x)>0\} \quad \mbox{and} \quad
\Omega^-:=\{x\in \Omega: a(x)<0\} $$ are nonempty, and that
$$ \Gamma:=\bar{\Omega}^+\cap \bar{\Omega}^-\subset \Omega,
\quad \mbox{with} \ \nabla a(x)\not= 0 \ \ \forall x\in \Gamma.
$$
As for $g(u)$, it is a $C^1$ function on $\mathbb R^+$ with power-like growth at infinity
\begin{equation}\lim_{s\to \infty}\frac{g(s)}{s^p}=l>0 \quad \mbox{for some} \
p>1.
\label{ggg}
\end{equation}
Without loss of generality, we may assume that $l=1$. The model (\ref{frac}) with $\alpha=2$ has been studied in \cite{BCN},
which is called the indefinite semilinear
 problem. Hence the case
$0<\alpha<2$ we studied here could be considered as an indefinite fractional elliptic problem.

The fractional Laplacian has attracted much attention recently. It
has applications in mathematical physics, biological modeling and
mathematical finances and so on. Especially, it appears in
turbulence and water wave, anomalous dynamics, flames propagation
and chemical reactions in liquids, population dynamics, geophysical
fluid dynamics, and American options in finance. It also has connections to conformal
geometry, e.g. \cite{CG}.

The fractional Laplacian $(-\lap)^{\frac{\alpha}{2}}$ in
$\mathbb R^n$ is a nonlocal operator defined as
$$(-\lap)^{\frac{\alpha}{2}}u= C_{n,\alpha} P.V. \int_{\mathbb
R^n}\frac{u(x)-u(y)}{|x-y|^{n+\alpha}} \,dy,$$ where $ P.V.$ means in
the cauchy principle value sense. Another equivalent definition is
given by Fourier transform, that is,
$$\widehat{(-\lap)^{\frac{\alpha}{2}}u}(\xi)=|\xi|^{\alpha}\hat{u}(\xi),$$
where $u(x)$ is in the Schwartz class of functions. Observe that the
above definitions are nonlocal. Recently, Caffarelli and Silvestre in \cite{CS}
introduced a local realization of the fractional Laplacian
$(-\lap)^{\frac{\alpha}{2}}$ in $\mathbb R^n$ through the
Dirichlet-Neumann map of an appropriate degenerate elliptic operator
in  $\mathbb R^{n+1}_+$. Based on Caffarelli-Silvestre's extension, we are able to study the
fractional Laplacian $(-\lap)^{\frac{\alpha}{2}}$ in a local way and
to use the tools for semilinear elliptic equations. More
precisely, if $u\in H^{\frac{\alpha}{2}}(\mathbb R^n)$, then
$w$ is its extension in $\mathbb
R^{n+1}_+$, if it solves the equation
\begin{equation}
\begin{array}{rll}
-div(y^{1-\alpha}\nabla w)=&0 \quad \quad &\mbox{in} \ \mathbb
R^{n+1}_+, \medskip \nonumber \\
w=&u \quad \quad   \quad \quad       &\mbox{on} \ \mathbb
R^{n}\times\{y=0\}. \nonumber
\end{array}
\end{equation}
It is shown that in \cite{CS} that
$$ \lim_{y\to 0^+} y^{1-\alpha} \frac{\partial w}{\partial
y}(x,y)=k_\alpha(-\lap)^{\frac{\alpha}{2}}u(x),$$ where the constant
$$k_\alpha=\frac{2^{1-\alpha}\Gamma(1-\frac{1}{2}\alpha)}{\Gamma(\frac{1}{2}\alpha)}.
$$
To define the fractional Laplacian in a bounded domain, the idea is
to use the Caffarelli-Silvestre's extension in a cylindrical domain. See \cite{CT} for the case $\alpha=1$ and
\cite{BCDS} for its generalization to $ 0<\alpha<2$.
Let $\{\lambda_k, \phi_k \}^\infty_{k=1}$ be the eigenvalues and
corresponding eigenfunctions of the Laplacian operator $-\lap$ in
$\Omega$ with zero Dirichlet boundary values on $\partial\Omega$,
\begin{equation}
\left \{
\begin{array}{rll}
-\lap \phi_k=&\lambda_k \phi_k \quad \quad &\mbox{in} \ \Omega, \medskip \\
\phi_k=&0 \quad \quad &\mbox{on} \ \partial\Omega
\end{array}
\right.
\end{equation}
such that $\|\phi_k\|_{L^2(\Omega)}=1$ and
$0<\lambda_1<\lambda_2\leq \lambda_3\leq \cdots $. The fractional
Sobolev space $H^{\frac{\alpha}{2}}_0(\Omega)$ is defined by
$$H^{\frac{\alpha}{2}}_0(\Omega)=\{u=\sum^{\infty}_{k=1} a_k\phi_k\in L^2(\Omega): \|u\|_{H^{\frac{\alpha}{2}}_0(\Omega)}
=(\sum a^2_k\lambda_k^\frac{\alpha}{2})^{\frac{1}{2}} \},$$ which is also a
Hilbert space. Let
$$C_\Omega=\{(x,y):x\in\Omega, y\in \mathbb R_+\}\subset \mathbb
R^{n+1}_+ $$ and $\partial_LC_\Omega:=\partial\Omega\times (0, \
\infty)$ its lateral boundary. If we reformulate the nonlocal
problem (\ref{frac}) by Caffarelli-Silvestre's extension, then it
corresponds to
\begin{equation}
\left \{
\begin{array}{rll}
-div(y^{1-\alpha}\nabla w)=&0 \quad \quad &\mbox{in} \ C_{\Omega},\medskip \nonumber  \\
w(x,y)=&0 \quad \quad &\mbox{on} \ \partial_LC_\Omega, \medskip\nonumber \\
w(x,0)=&u(x) \quad \quad &\mbox{on} \ \Omega\times\{0\}, \medskip\nonumber \\
-\lim\limits_{y\to 0^+}y^{1-\alpha}\frac{\partial w}{\partial y} =
&a(x)g(u) \quad \quad   &\mbox{on} \ \Omega\times \{0\}. \medskip\nonumber
\end{array}
\right.
\end{equation}

The indefinite elliptic problem for Laplace operator
\begin{equation}
\left \{
\begin{array}{rll}
-\lap u= &a(x)u^p \quad \quad &\mbox{in} \ \Omega, \medskip \\
u=&0 \quad \quad &\mbox{on} \ \partial\Omega \label{semi}
\end{array}
\right.
\end{equation}
has been extensively studied in the literature, where $a(x)$
satisfies the condition as above. See e.g. \cite{AL}, \cite{AT}, \cite{BCN}, \cite{BCN1},
\cite{DL}, \cite{Z}, just to mention
a few. In order to prove the existence and multiplicity of positive
solutions, it is very important to obtain a priori bound of
solutions. Blow-up techniques of Gidas-Spruck \cite{GS} and Liouville
theorems are very useful in obtaining the a priori bound. Concerning
 problem (\ref{semi}), the maxima of a sequence of solutions may
blow up on $\partial \Omega$, $\Omega^+\cup \Omega^-$ or $\Gamma$.
If the blow-up occurs at $\partial \Omega$ or $\Omega^+\cup
\Omega^-$, we can make use of the classical Liouville theorems in
$\mathbb R^n$ and $\mathbb R^n_+$ to get a contradiction and hence obtain a priori bound. If the blow-up
occurs on $\Gamma$, Berestycki, Capuzzo-Docetta and Nirenberg in \cite{BCN} were
able to obtain the a priori bound by establishing a
Liouville theorem for
\begin{equation}
\left \{
\begin{array}{lrl}
-\lap u= x_1u^p \quad \quad &\mbox{in} \ \mathbb R^n, \medskip \\
u\geq 0 \quad \quad &\mbox{in} \ \mathbb R^n. \label{inde}
\end{array}
\right.
\end{equation}
It was shown that there exists no positive solution for
$p<\frac{n+2}{n-1}$. Later Chen and Li \cite{CL},\cite{CL1} further relaxed the
restriction on $a(x)$ near $\Gamma$ and obtain a priori bound with a
general $p>1$. As we know, Liouville theorem is a key in obtaining a
priori bound. There are also several Liouville theorems for
indefinite elliptic problems. Lin \cite{Lin} proved that the nonnegative
solution for
$$ -\lap u=x_1^mu^{n^\ast} \quad \quad \mbox{in} \ \mathbb R^n $$
is trivial, when $m$ is an odd positive integer and
$n^\ast=\frac{n+2}{n-2}$ is the critical exponent of Sobolev
imbedding.  Du and Li \cite{DL} considered nonnegative solution of the
problem
\begin{equation}
\left \{
\begin{array}{lll}
-\lap u=h(x_1) u^p \quad \quad &\mbox{in} \ \mathbb R^n, \medskip \\
\sup_{\mathbb R^n}<\infty \quad \quad  &\mbox{in} \ \mathbb R^n,
\end{array}
\right.
\end{equation}
where $h(t)=t|t|^{s}$ or $h(t)=(t^+)^{s}$ for some $s>0$ and $p>1$.
They showed the solution is trivial. Zhu \cite{Z1} investigated the indefinite nonlinear
boundary condition motivated by a prescribing sign-changing scalar
curvature problem on compact Riemannian manifolds with boundary.  He
proved that there exists no solution for
\begin{equation}
\left \{
\begin{array}{rll}
-\lap u=&0 \quad \quad &\mbox{in} \ \mathbb R^n,  \nonumber \medskip \\
\frac{\partial u}{\partial x_n}=&-x_1u^p \quad \quad &\mbox{on} \
\partial \mathbb R^n_+.
\end{array}
\right.
\end{equation}

If one considers the fractional indefinite problem (\ref{frac}) and applies the blow-up technique,
one will also have to deal with the case that the blow-up occurs on $\Gamma$. Hence we shall first establish a
Berestycki, Capuzzo-Docetta and Nirenberg's type Liouville theorem for the fractional Laplacian.
This is our first goal. We prove
\begin{theorem}
Let $1\leq \alpha<2$ and $w\in H^1_{loc}(y^{1-\alpha},
\overline{R^{n+1}_+})$. Then the degenerate elliptic equation
\begin{equation}
\left \{
\begin{array}{rll}
- div(y^{1-\alpha} \nabla w) =& 0 \quad \quad &\mbox{in} \ \mathbb
R^{n+1}_+,  \medskip
\\
-\lim\limits_{y\to 0^+} y^{1-\alpha}\frac{ \partial w}{\partial y}(x,
y)=&x_1 w^p(x,0) \quad \quad &\mbox{on} \
\partial\mathbb R^{n+1}_+
\label{mai}
\end{array}
\right.
\end{equation}
has no positive bounded solution provided $1<p<\infty$. \label{th1}
\end{theorem}
The weighted Sobolev space in a domain $D$ in Theorem 1 is given by
$$ H^1(y^{1-\alpha}, D):=\{ u\in L^2(D, y^{1-\alpha}dxdy):
\ \ |\nabla u|\in L^2(D, y^{1-\alpha}dxdy)\}. $$

An immediate
consequence of the theorem is the following.
\begin{corollary}
Let $1\leq \alpha<2$  and $u$  be nonnegative bounded solution of
\begin{equation}
(-\lap)^{\frac{\alpha}{2}} u=x_1u^p \quad \quad \mbox{in} \ \mathbb
R^n.
\end{equation}
Then $u\equiv 0$ if $1<p<\infty$.
\end{corollary}

With aid of the Liouville theorem for (\ref{mai}), we are able to establish a universal $L^\infty$ bound
for every solution in (\ref{frac}).
\begin{theorem}
If $1\leq \alpha<2$ and $1<p<\frac{n+\alpha}{n-\alpha}$, then there exists a universal positive constant $C$ such that every solution
in (\ref{frac}) satisfies
$$\|u\|_{\infty} \leq C. $$
\label{th2}
\end{theorem}
Due to the importance and powerful applications of Liouville theorems in elliptic problems, we further investigate the
Liouville theorem for the equation
\begin{equation}
\left \{
\begin{array}{rll}
- div(y^{1-\alpha} \nabla w) =& 0 \quad \quad &\mbox{in} \ \mathbb
R^{n+1}_+, \medskip\\
\lim\limits_{y\to 0^+} y^{1-\alpha}\frac{ \partial w}{\partial y}(x,
y)=& w^p(x,0) \quad \quad &\mbox{on} \
\partial\mathbb R^{n+1}_+.
\label{new}
\end{array}
\right.
\end{equation}
We are able to classify all the solutions.

\begin{theorem}
Let $w\in H^1_{loc}(y^{1-\alpha},
\overline{R^{n+1}_+})$ be a nonegative solution of (\ref{new}), $0<\alpha<2$, and $p>1$, then $w=\frac{ay^\alpha}{\alpha}+b$
where $b>0$ and $ a=b^p$.
\label{th3}
\end{theorem}

 A similar result for the case of $\alpha=1$ has been obtained in \cite{LZ}. Equation (\ref{new}) can be regarded as the
Caffarelli-Silvestre's extension for
\begin{equation}
(-\lap)^{\frac{\alpha}{2}} u=-u^p \quad \quad \mbox{in} \ \mathbb
R^n.
\label{nonge}
\end{equation}
We know that there exist no nonnegative solution for (\ref{nonge}). Interestingly, the solutions for
(\ref{new}) exist locally and we are able to classify all.

The paper is organized as follows. Section 1 is devoted to proving Liouville theorem for (\ref{mai}).
We study the a priori bound for the fractional indefinite elliptic problem (\ref{frac}) in section 2.
In section 3, we classify all the solutions in (\ref{new}) and establish Theorem \ref{th3}.

\section{Liouville theorem for fractional indefinite problem}

An efficient way to prove the Liouville theorem is to apply the moving
plane method in appropriate settings. We show Theorem \ref{th1} by
a contradiction argument and adapt the ideas in \cite{Z1} and \cite{Lin}. The case of $\alpha=1$ is studied in \cite{Z1}
 by the method of moving planes. However,
for the general case $1\leq \alpha<2$, we have to introduce new and appropriate auxiliary functions and
take the regularity of solutions into considerations.
New ideas are introduced on selecting the auxiliary functions and more complicated calculations
are involved.

Suppose that there is a nontrivial solution to
(\ref{mai}). By the strong maximum principle in \cite{CaS}, we know that
$w(x, y)>0$ in $\mathbb R^{n+1}_+$. For ease of notation, we define
the operators
$$L_\alpha w:=y^{\alpha-1}div(y^{1-\alpha}\nabla w)=\lap
w+\frac{1-\alpha}{y}\frac{\partial w}{\partial y} $$ and
$$ \frac{\partial w}{\partial \nu^{\alpha}}:=-\lim\limits_{y\to 0^+} y^{1-\alpha}\frac{ \partial w}{\partial y}.$$
Then problem (\ref{mai}) can be rewritten in the form:
\begin{equation}
\left \{
\begin{array}{rll}
L_\alpha w =& 0 \quad \quad &\mbox{in} \ \mathbb
R^{n+1}_+, \medskip \\
\medskip
\frac{\partial w}{\partial \nu^{\alpha}}=&x_1 w^p \quad \quad
&\mbox{on} \
\partial\mathbb R^{n+1}_+.
\label{mai1}
\end{array}
\right.
\end{equation}

To employ the moving plane method, we set up some useful notation.
Let $$X=(x, y)=(x_1, \cdots, x_n, y).$$ For $\lambda\in \mathbb R$, we
define
$$\Sigma_{\lambda}=\{ X\in \mathbb R^{n+1}_+| x_1<\lambda\}
$$ and
$$T_\lambda=\{X\in \mathbb R^{n+1}_+|x_1=\lambda\}.$$
The reflection of $X$ with respect to $T_\lambda$ is
$$X^{\lambda}=(2\lambda-x_1, x_2,\cdots, x_n, y).$$
Set $w_\lambda(X)=w(X^\lambda)$. We compare the value of $w$ and
$w_\lambda$. Let
$$v_\lambda(X)=w_\lambda(X)-w(X).$$
We can verify that $v_\lambda(X)$ satisfies
\begin{equation}
\left \{
\begin{array}{rll}
L_\alpha v_\lambda =& 0 \quad \quad &\mbox{in} \ \Sigma_\lambda,  \medskip \\
\frac{\partial v_\lambda}{\partial \nu^{\alpha}}=&x_1^\lambda
w_\lambda^p-x_1 w^p \quad \quad &\mbox{on} \
\partial\mathbb R^{n+1}_+\cap \overline{\Sigma_\lambda}.
\label{cal}
\end{array}
\right.
\end{equation}

We would like to show that $w(X)$ is monotone nondecreasing in $x_1$
direction. The goal is to prove the following proposition.

\begin{proposition}
For any $\lambda\in \mathbb R$ and $X\in \Sigma_\lambda$,
$v_\lambda(X)\geq 0$ in $\Sigma_\lambda$. \label{pro1}
\end{proposition}
\begin{proof}
We divide the proof in two major steps.

\medskip
\emph{Step 1:} For any $\lambda\leq 0$ and $X\in \Sigma_\lambda$,
$v_\lambda(X)\geq 0$ in $\Sigma_\lambda$.
\medskip

We select the test function $$g(X)=\sum^{n}_{i=2} \ln
[(2-x_1)^2+x_i^2]+\ln[(2-x_1)^2+y^2].$$ It is easy to see that
$$g(X)\to \infty \quad \mbox{as} \ |X|\to \infty.$$ We can also check
that for $\lambda\leq 0$,
\begin{equation}\lap g(X)=0, \quad g(X)>0, \; \frac{\partial g}{\partial y} > 0 \;
\ \mbox{in} \ \Sigma_\lambda  \quad \mbox{and} \ \ \frac{\partial
g}{\partial y}=0 \ \mbox{on} \
\partial \mathbb R^{n+1}_+\cap \overline{\Sigma_\lambda} .
\label{fac}
\end{equation}

Define $$\bar v_\lambda(X)=\frac{v_\lambda(X)}{g(X)}.$$ Then $\bar
v_\lambda$ satisfies the following equation

\begin{equation}
\left \{
\begin{array}{lll}
 L_\alpha\bar v_\lambda+2 \frac{\nabla g}{g}\cdot \nabla \bar v_\lambda+\frac{1-\alpha}
{y}\frac{1}{g}\cdot\frac{\partial g}{\partial y}\cdot\bar v_\lambda  = 0 \quad \quad &\mbox{in} \ \Sigma_\lambda, \medskip\\
\frac{\partial \bar v_\lambda}{\partial \nu^{\alpha}}=
\frac{1}{g}\cdot (x_1^\lambda w_\lambda^p-x_1 w^p) \quad \quad
&\mbox{on} \
\partial\mathbb R^{n+1}_+\cap \overline{\Sigma_\lambda}.
\label{call}
\end{array}
\right.
\end{equation}
\end{proof}
In the above equation, we have used the facts in (\ref{fac}). Since $w$ is bounded, then $\bar
v_\lambda(X)\to 0$ as $|X|\to \infty$. If $\bar v_\lambda(X)<0$ at
some point in $\Sigma_\lambda$, then $\bar
v_\lambda(X^0)=\inf_{\Sigma_\lambda}\bar v_\lambda(X)$ must be
attained in $\overline{\Sigma_\lambda}$. If $X^0\in \Sigma_\lambda$,
then $$\lap \bar{v}_\lambda (X^0)\geq 0, \ \ \nabla \bar v_\lambda(X^0)=0,
\ \ \mbox{and} \ \ \bar v_\lambda(X^0)<0.$$ It contracts the first equation in
(\ref{call}) since $1\leq \alpha<2$. It is also impossible for $X^0\in
T_\lambda$. So $X^0\in \partial \mathbb R^{n+1}_+\cap
\{x_1<\lambda\}$, which implies that $\frac{\partial \bar
v_\lambda}{\partial y}(X^0)\geq 0$. Hence
$$\frac{\partial \bar v_\lambda}{\partial \nu^{\alpha}}\leq 0.$$
However, $$ x_1^{0, \lambda} w^p_\lambda(X^0)-x_1^0 w^p(X^0)> x_1^0(
w^p_\lambda(X^0)-w^p(X^0))>0, $$ where $x_1^{0, \lambda}$ is the
reflection of $x_1^0$ with respect to $T_\lambda$. Obviously a
contradiction is arrived because of the second equation in
(\ref{call}). Therefore $\bar v_\lambda(X)\geq 0$ in
$\Sigma_\lambda$ for $\lambda\leq 0$, so is $v_\lambda$. We complete
the first step.

We move the plane further to the right. Define
$$\lambda_0=\sup \{\lambda : v_\lambda(X)\geq 0 \ \mbox{in} \ \Sigma_\mu \ \mbox{for} \ \mu\leq \lambda\}.$$
From the conclusion in step 1, we know that $\lambda_0\geq 0$. We shall show that the plane can be moved all the
way to the positive infinity, that is,

\medskip
\emph{ Step 2}: $\lambda_0=+\infty.$
\medskip

We also prove it by contradiction. Suppose that $\lambda_0<+\infty$.
It is clear that $v_{\lambda_0}(X)\not \equiv 0$. By the  maximum
principle in \cite{CaS}, we infer that \begin{equation} v_{\lambda_0}(X)>0
\quad \mbox{in} \ \Sigma_{\lambda_0} \label{back}
\end{equation}
and
$$ \frac{\partial v_{\lambda_0}}{\partial x_1}(X)<0 \quad \mbox{for}
\ X\in T_{\lambda_0}\cap \Sigma_{\lambda_0}. $$

In order to derive a contradiction in the future argument, a subtle
analysis has to be given to the corner point $\hat{X}\in
T_{\lambda_0}\cap
\partial \mathbb R^{n+1}_+$. We are able to establish the following technical
lemma to take care of $\hat{X}$. Our argument is inspired by lemma 2.4 in \cite{LYZ}.
Since we consider degenerate elliptic equations, more considerations have to be taken into the choice
of auxiliary function. After careful calculations, we are able to find the desired auxiliary function in
fractional Laplacian setting.

\begin{lemma}
Assume that $v_{\lambda_0}$ satisfies (\ref{cal}) and
$v_{\lambda_0}>0$ in $\Sigma_{\lambda_0}.$ If $\hat{X}\in
T_{\lambda_0}\cap\partial \mathbb R^{n+1}_+$, then
$$\frac{\partial v_{\lambda_0}}{\partial x_1}(\hat{X})<0.$$
\label{tech}
\end{lemma}
\begin{proof}
Without loss of generality, we assume that $\lambda_0=1$ and
$\hat{X}=(1,0,\cdots, 0)$. Set
$$\hat{\Omega}:=\{X: \mathbb B_1^+\backslash \overline{\mathbb B^+_{\frac{1}{2}}}, \ y\leq \frac{1}{5}\}.$$
Introduce the function
$$ h(X)=\beta(|x|^{-\gamma}-1)(y^{\alpha}+\mu),$$
where $\gamma>\max\{\frac{n+\alpha}{2}, \ (n-2) \}$ and $\beta,
\mu>0$ will be determined later. The test function is given by
$$\psi(X)=h(X)-|X|^{\alpha-n}h(\frac{X}{|X|^2}).$$
Let $$Z(X)=|X|^{\alpha-n}w(\xi),$$ where $\xi:=\frac{X}{|X|^2}$.

A matter of calculus yields that
\begin{equation} L_\alpha
Z(X)=|X|^{\alpha-n-4} L_\alpha w(\xi). \label{matt}
\end{equation}
Direct calculations also show that
$$ L_\alpha h(X)=\beta \gamma
(\gamma+2-n)|x|^{-\gamma-2}(y^\alpha+\mu).$$ It follows from
(\ref{matt}) that
$$ L_\alpha (|X|^{\alpha-n}h(\frac{X}{|X|^2}))=\beta \gamma
(\gamma+2-n)|X|^{\alpha-n+2\gamma}|x|^{-\gamma-2}(|X|^{-2\alpha}y^{\alpha}+\mu).
$$
Therefore,
\begin{eqnarray}
L_\alpha \psi(X)&=&\beta \gamma
(\gamma+2-n)|x|^{-\gamma-2}(y^\alpha+\mu-y^\alpha|X|^{2\gamma-\alpha-n}-|X|^{2\gamma+\alpha-n}\mu)\nonumber \\
&\geq & 0
\end{eqnarray}
since  $\gamma>\max\{\frac{n+\alpha}{2}, \ (n-2) \}$ and
$\frac{1}{2}<|X|<1$. Let $$A(X)=v_{\lambda_0}(X)-\psi(X).$$ It follows
that
\begin{equation}
\left \{
\begin{array}{rll}
L_\alpha A(X) \leq & 0 \quad \quad &\mbox{in} \ \hat{\Omega},\medskip  \\
\frac{\partial A(X)}{\partial \nu^{\alpha}}\geq &x_1(
w_{\lambda_0}^p-w^p)+\frac{\partial \psi}{\partial \nu^{\alpha}} \quad
\quad &\mbox{on} \ \ \hat{\Omega}\cap \partial\mathbb R^{n+1}_+.
\label{tes}
\end{array}
\right.
\end{equation}
By choosing suitable small $\beta$ and $\mu$, we want to show that
\begin{equation}
A(X)\geq 0 \quad \forall X\in \hat{\Omega}. \label{aim}
\end{equation}
Since $v_{\lambda_0}>0$ in $\Sigma_{\lambda_0}$, we can find some positive constant
$\beta_0$ such that
$$ A(X)\geq 0 \quad \mbox{on} \ \partial\hat{\Omega}\cap
\{\partial\mathbb B_{\frac{1}{2}}\cup \{y=1/5\}\} $$
 for
all $0<\beta<\beta_0$. By the construction of $\psi$, we know that
$\psi(X)=0$ on $\partial\mathbb B_1$. Then $A(X)>0$ on
$\partial\hat{\Omega}\cap\partial\mathbb B_1$. If (\ref{aim}) is not
true, then there exist some $\bar X=(\bar x, \bar y)$ such that
$$ A(\bar X)=\min_{\hat{\Omega}} A(X)<0.$$
By the maximum principle, it yields that $\bar X\in \hat{\Omega}\cap
\partial\mathbb R^{n+1}_+$, that is, $\bar y=0$. We have
\begin{equation}
v_{\lambda_0}(\bar X)<\beta\mu(|\bar x|^{-\gamma}-1)(|\bar
x|^{\gamma+\alpha-n}+1) \label{min}
\end{equation}
and
\begin{equation}
\frac{\partial A}{\partial y}(\bar X)\geq 0, \ \ i.e.  \ \
\frac{\partial A}{\partial \nu^{\alpha}}(\bar X)\leq 0. \label{kao}
\end{equation}
If we calculate $\frac{\partial\psi}{\partial \nu^\alpha}$ at
$\bar X$, it is shown that
\begin{equation}
\frac{\partial\psi}{\partial \nu^\alpha}(\bar X)=\alpha\beta(|\bar
x|^{-\gamma}-1)(1+|\bar x|^{\gamma-\alpha-n}). \label{non}
\end{equation}
The inequality (\ref{kao}) and equality (\ref{non}) together with
the second inequality in (\ref{tes}) implies that
$$ p\bar x_1 \xi^{ p-1}(\bar X) v_{\lambda_0}(\bar X)+\alpha\beta(|\bar
x|^{-\gamma}-1)(1+|\bar x|^{\gamma-\alpha-n})\leq 0,   $$ where
$\xi(\bar X)$ is between $w_{\lambda_0}(\bar X)$ and
$w(\bar X)$.

Thanks to (\ref{min}), it follows that
\begin{equation}
\mu p\bar x_1 \xi^{ p-1}(\bar X)+\alpha\leq 0. \label{fin}
\end{equation}
 If we choose $$0<\mu<\min_{\frac{1}{2}\leq |X|\leq
1}\frac{\alpha}{1+p|x|\xi^{p-1}(X)}$$ at the beginning, we will arrive at a
contradiction from (\ref{fin}). Therefore $A(X)\geq 0$ in
$\hat{\Omega}$. Note that $A(\hat{X})=0$, which implies that
$$ \frac{\partial A}{\partial x_1}(\hat{X})\leq 0.$$ Thus
\begin{eqnarray}
\frac{\partial v_{\lambda_0}(\hat{X})}{\partial x_1}&=&
\frac{\partial
A(\hat{X})}{\partial x_1}+ \frac{\psi(\hat{X})}{\partial x_1} \nonumber \medskip \\
&\leq &\frac{\psi(\hat{X})}{\partial x_1}=-2\alpha\beta\mu<0.
\nonumber
\end{eqnarray}
This completes the proof of the lemma.
\end{proof}

We now continue our argument in Step 2. Let $$g(X)=\sum^{n}_{i=2} \ln
[(2+\lambda_0-x_1)^2+x_i^2]+\ln[(2+\lambda_0-x_1)^2+y^2] $$and $$\bar
v_\lambda(X)=\frac{v_\lambda(X)}{g(X)}.$$ We know that $\bar
v_\lambda(X)$ satisfies (\ref{call}) in $\Sigma_{\lambda_0}$. For
a fixed $\lambda$, it is true that $\bar v_\lambda(X)\to 0$ as $|x|\to
\infty$. From the definition of $\lambda_0$ and the maximum
principle, there exist sequences of $\lambda_j>\lambda_0$ and
$X^j\in \overline\Sigma_{\lambda_j}\cap \partial \mathbb R^{n+1}_+
$, such that $\lambda_j\to \lambda_0$ as $j\to \infty$ and
$$ \bar v_{\lambda_j}(X^j)=\inf_{\Sigma_{\lambda_j}} \bar  v_{\lambda_j}(X)<0. $$
Moreover,
\begin{equation}
\frac{\partial \bar v_{\lambda_j}}{\partial x_1}(X^j)=0 \quad
\mbox{and} \quad \frac{\partial \bar v_{\lambda_j}}{\partial
\nu^{\alpha}}(X^j)\leq 0. \label{mor}
\end{equation}
However, we are able to show that (\ref{mor}) contradicts the
following lemma. To complete the proof of Proposition \ref{pro1}, we only need to
prove the following result.

\begin{lemma}
Let $\lambda_j>\lambda_0$ and $\lambda_j\to \lambda_0$ as $j\to
\infty$. If $X^j\in \Sigma_{\lambda_j}\cap  \partial \mathbb
R^{n+1}_+$ with
$$\frac{\partial \bar v_{\lambda_j}}{\partial x_1}(X^j)=0 \quad \mbox{and} \quad \bar v_{\lambda_j}(X^j)<0,$$
then there exists a sufficiently large $j_0$ such that
$$\frac{\partial \bar v_{\lambda_j}}{\partial
\nu^{\alpha}}(X^j)>0 \quad \forall j>j_0.  $$ \label{lemm2}
\end{lemma}

\begin{proof}
We verify it by contradiction. If it is not true, then there exist
sequences $\lambda_j\to \lambda_0 \ (\lambda_j>\lambda_0)$, $X^j\in
\Sigma_{\lambda_j}\cap  \partial \mathbb R^{n+1}_+$ and $\bar
v_{\lambda_j}(X^j)$ such that
\begin{equation}\frac{\partial \bar
v_{\lambda_j}}{\partial x_1}(X^j)=0 \quad \mbox{and} \quad \bar
v_{\lambda_j}(X^j)<0, \label{ast}
\end{equation}
but
\begin{equation}
\frac{\partial \bar v_{\lambda_j}}{\partial \nu^{\alpha}}(X^j) \leq
0. \label{ast1}
\end{equation}
Let $x_1^j$ be the first component of $X^j$, and then $x_1^{j,
\lambda_j}=2\lambda_j-x_1^j$. Because of (\ref{ast1}),
$$x_1^{j, \lambda_j}w_{\lambda_j}^p(X^j)-x_1^{j}w^p(X^j)\leq 0.$$
Due to the fact that $\lambda_0\geq 0$, we have $x_1^{j, \lambda_j}\geq
0$. Thus, $x_1^j>0$. Furthermore, $x_1^j\in (0, \ \lambda_j)$. Note
that $ X^j\in \partial\mathbb R^{n+1}_+. $ We may write
$X^j=(x_1^j, \ \tilde{x}^j, \ 0)$, here  $\tilde{x}^j=(x_2^j,\cdots,
x_n^j)$. If $\tilde{x}^j$ is bounded, there exists $X^0$ such that
there is a subsequence (we do not distinguish the sequence and its
 subsequence in the whole paper ) such that $X^j\to X^0 \in
\overline\Sigma_{\lambda_0}\cap
\partial \mathbb R^{n+1}_+$. Also notice that
$$\frac{\partial \bar v_{\lambda_0}}{\partial x_1}(X^0)=\lim\limits_{j\to\infty}\frac{\partial
\bar v_{\lambda_j}}{\partial x_1}(X^j)=0.$$ From Lemma \ref{tech},
we can infer that $X^0\not\in T_{\lambda_0}$ and $\bar
v_{\lambda_0}(X^0)\leq 0$, which contradicts the fact that $$ \bar
v_{\lambda_0}(X^0)>  0 \quad \mbox{in} \ \
\overline\Sigma_{\lambda_0}\backslash T_{\lambda_0}.$$
Without loss of generality, we assume that $\tilde{x}^j$ is
unbounded, that is, $|\tilde{x}^j|\to \infty$ as $j\to\infty$. Let
$$ w^j(X)=w(x_1, \tilde{x}^j+\tilde{x}, y), $$
where $\tilde{x}=(x_2,\cdots, x_n)$. Since $w^j$ is uniformly
bounded, by Corollary 2.1 in \cite{JLX}),
there exists
$\tilde{w}\in H^1_{loc}(y^{1-\alpha}, \mathbb R^{n+1}_+)\cap C^{\beta}_{loc}(
\overline{\mathbb R^{n+1}_+})$ such that
\begin{equation}
\left \{
\begin{array}{lll}
w^j\rightharpoonup \tilde{w}\quad &\mbox{weakly in} \ H^1_{loc}(y^{1-\alpha}, \overline{\mathbb R^{n+1}_+}),
\nonumber \medskip \\
w^j\to \tilde{w} \quad & \mbox{in} \ C^{0, \ \beta}_{loc}(
\overline {R^{n+1}_+}) \nonumber
\end{array}
\right.
\end{equation}
for $\beta >0$, and
$\tilde{w}$ satisfies
\begin{equation}
\left \{
\begin{array}{rll}
- L_\alpha \tilde{w} =& 0 \quad \quad &\mbox{in} \ \mathbb R^{n+1}_+, \medskip\\
\tilde{w}>&0 \quad \quad &\mbox{in} \ \mathbb R^{n+1}_+,
\medskip \\
\frac{\partial\tilde{w}}{\partial \nu^{\alpha}}=&x_1 \tilde{w} \quad
\quad &\mbox{on} \
\partial\mathbb R^{n+1}_+.
\label{why}
\end{array}
\right.
\end{equation}
Next we want to show that $\tilde{w}\equiv 0$ in $\mathbb
R^{n+1}_+$.\\

Let $\tilde v_{\lambda}=\tilde{w}_{\lambda}-\tilde{w}$. Then $\tilde
v_{\lambda}$ satisfies
\begin{equation}
\left \{
\begin{array}{rll}
- L_\alpha \tilde v_{\lambda} =& 0 \quad \quad &\mbox{in} \ \mathbb R^{n+1}_+, \medskip \\
\frac{\partial \tilde v_{\lambda}}{\partial \nu^{\alpha}}=&x_1^\lambda
\tilde{w}_\lambda^p-x_1 \tilde{w}^p \quad \quad &\mbox{on} \
\partial\mathbb R^{n+1}_+.
\label{whyy}
\end{array}
\right.
\end{equation}
By (\ref{back}), we know that $\tilde{v}_{\lambda_0}(X)\geq 0$ in
$\Sigma_{\lambda_0}$. From (\ref{ast}), it follows that
$$\tilde{v}_{\lambda_j}(x^j_1, 0, 0)\leq 0.
$$
Since  $x_1^j\in (0, \ \lambda_j)$, then $x_1^j$ will converge to
some number in $[0 \ \lambda_0]$ as $j\to\infty$. If $x_1^j\to
x^0_1<\lambda_0$, we have $\tilde v_{\lambda_0}(x^0_1, 0, 0)=0$.
However,
\begin{eqnarray}
\frac{\partial \tilde{v}_{\lambda_0}}{\partial \nu^{\alpha}}(x^0_1,
0, 0)&=& x_1^{0, \lambda_0}\tilde{w}_{\lambda_0}^p(x_1^0, 0,
0)-x_1^{0}\tilde{w}^p(x_1^0, 0, 0) \nonumber \medskip \\
&>& x_1^0 (\tilde{w}_{\lambda_0}^p(x_1^0, 0, 0)-\tilde{w}^p(x_1^0,
0, 0)) \nonumber \medskip \\
&=& 0,
\end{eqnarray}
 which contradicts the Hopf lemma in \cite{CaS}. Thus, $x_1^j\to
x^0_1=\lambda_0$ as $j\to\infty$. If $\tilde{w}\not\equiv 0$, then
$\tilde{v}_{\lambda_0}>0$ in $\Sigma_{\lambda_0}$. Since
$$ \frac{\partial v_\lambda}{\partial x_1}=g\frac{\partial \bar v_\lambda}{\partial
x_1}+\frac{\partial g}{\partial x_1}\bar v_\lambda,$$ it follows
from (\ref{ast}) that
\begin{eqnarray}
\frac{\partial \tilde{v}_{\lambda_0}}{\partial x_1}(x^0, 0,
0)&=&\lim\limits_{j\to\infty}\frac{\partial v_{\lambda_j}}{\partial
x_1}(X^j) \nonumber \medskip\\
&=& \lim\limits_{j\to\infty}\frac{\partial \bar
v_{\lambda_j}}{\partial x_1}(X^j)g+\lim\limits_{j\to\infty}\bar v_{\lambda_j}(X^j)\frac{\partial g}{\partial
x_1} \nonumber \medskip \\ &=&
0. \nonumber
\end{eqnarray}
It contradicts Lemma \ref{tech}. Therefore, $\tilde{w}\equiv 0$.

It follows that $w^j(0)\to 0$ as $j\to\infty$. Define
$$W^j(X)=\frac{w^j(X)}{w^j(0)}.$$
By the Harnack inequality in \cite{CaS}, $W^j(X)$ is bounded in
$\bar{\mathbb B}_R^+$ for every $R$. Also $W^j(X)$ satisifies
\begin{equation}
\left \{
\begin{array}{rll}
- L_\alpha W^j =& 0 \quad \quad &\mbox{in} \ \mathbb R^{n+1}_+, \medskip\\
\medskip
\frac{\partial{W^j}}{\partial \nu^{\alpha}}=&x_1(W^j)^{p}(w^j(0))^{p-1} \quad \quad &\mbox{on}
\
\partial\mathbb R^{n+1}_+.
\end{array}
\right.
\end{equation}
By Corollary 2.1 in \cite{JLX} again, it follows that
\begin{equation}
\left \{
\begin{array}{lll}
W^j\rightharpoonup {W}\quad &\mbox{weakly in} \ H^1_{loc}(y^{1-\alpha}, \overline{\mathbb R^{n+1}_+}), \nonumber
\medskip \\
W^j\to {W} \quad & \mbox{in} \ C^{0,\ \beta}_{loc}(
\overline {R^{n+1}_+}) \nonumber
\end{array}
\right.
\end{equation}
for some $\beta>0$ with $W(0)=1$. Furthermore, $W(X)\geq 0$ satisfies
\begin{equation}
\left \{
\begin{array}{rll}
- L_\alpha W =& 0 \quad \quad &\mbox{in} \ \mathbb R^{n+1}_+, \medskip \\
\medskip
\frac{\partial{W}}{\partial \nu^{\alpha}}=&0 \quad \quad &\mbox{on}
\
\partial\mathbb R^{n+1}_+.
\label{kkk}
\end{array}
\right.
\end{equation}
Thanks to the Harnack inequality in \cite{CaS} again, that is,
$$ \sup_{\mathbb B^+_R} W\leq C\inf_{\mathbb B^+_R} W,$$
where $C$ is independent of $R$. Let $$\bar W=W-\inf W.$$  $\bar W$ satisfies the same equation as
(\ref{kkk}). Since $\inf \bar W=0$, then exists a sequence of $X_j\in\partial\mathbb R^{n+1}_+$ such that
$\lim\limits_{j\to \infty}\bar W(X_j)=0$. Thus, for every $\epsilon>0$, there exist some $\bar X\in \partial\mathbb R^{n+1}_+$
such that $ \bar W(\bar X)\leq \epsilon$. By Harnack inequality again,
$$ \sup_{\mathbb B^+_R(\bar X)} \bar W (X)\leq C\epsilon.    $$
Since $C$ is independent of $R$,
$$ \bar W (X)\leq C\epsilon \quad \forall X\in \mathbb R^{n+1}_+. $$
Let $\epsilon\to 0$, then $\bar W\equiv$ consant. Hence $W(X)\equiv 1$. By Proposition 2.6 in \cite{JLX}, it
follows that for every $R$, $\epsilon$ and $X\in \mathbb
B_R^+(x_1^j, 0, 0),$ $$ |\nabla_x W^j(X)|\leq \epsilon, \ \  \mbox{as} \ j \
\mbox{large enough}.$$ The application of Harnack inequality further
implies that
$$ |\nabla_x w^j(X)|\leq \epsilon w^j(0)\leq C\epsilon \inf_{\mathbb
B_R^+(x_1^j, 0, 0)}w^j(X)$$ for $X\in \mathbb B_R^+(x_1^j, 0, 0)$
and large enough $j$, where $C=C(R, \max(w), \lambda_0).$ Therefore,
\begin{equation}
|\nabla_x w(X)|\leq C\epsilon w(X) \quad \mbox{for} \ X= (t,
\tilde{x}^j, 0), \label{alm}
\end{equation}
where $t\in (0, \ \lambda_0+R/2)$ with large $R$.

For $X=(t, \tilde{x}^j, 0)$ and $t\in (0, \ \lambda_0+R/2)$,  from
(\ref{alm}),
\begin{eqnarray}
\frac{\partial (x_1w^p)}{\partial
x_1}&=&w^p+px_1w^{p-1}\frac{\partial w}{\partial x_1}\nonumber \medskip  \\
&=& w^{p-1}(w+px_1\frac{\partial w}{\partial x_1}) \nonumber \medskip  \\
&> &0 \nonumber \medskip
\end{eqnarray}
if $\epsilon$ is sufficiently small. Then, for $j$ large enough,
$$\frac{\partial \bar v_{\lambda_j}}{\partial \nu^{\alpha}}(X^j) =
\frac{1}{g}\cdot (x_1^{j,
\lambda_j}w_{\lambda_j}^p(X^j)-x_1^{j}w^p(X^j))>0
$$
which contradicts (\ref{ast1}). We finally arrive at the conclusion
of Lemma \ref{lemm2}. This completes the proof of Proposition
\ref{pro1}.

\end{proof}

With the help of Proposition \ref{pro1}, we are able to give the
proof of Theorem \ref{th1}. We need to construct some new type of auxiliary function. Unlike the
semilinear Laplacian equation, the construction of auxiliary functions for fractional Laplacian is more involved.
Our auxiliary function is base on the product of the first eigenfunction of Laplacian equation and some Bessel function.

\begin{proof}[Proof of Theorem 1]
We first introduce some test functions.  $\psi$ solves the following equation
\begin{eqnarray}
\psi^{''}+\frac{1-\alpha}{s}\psi'&=&\psi, \nonumber \medskip \\
\psi(0)&=&1, \nonumber \medskip \\
\lim\limits_{s\to \infty} \psi(s)&=&0.
\end{eqnarray}
In fact $\psi$ minimizes the following function
$$ H_\alpha(\psi):=\int_{0}^{\infty} (|\psi(s)|^2+|\psi'(s)|^2)s^{1-\alpha}\,ds
.$$ It is known that $\psi$ is a combination of Bessel function \cite{L} or \cite{BCDS}
and it satisfies the following asymptotic behavior
\begin{equation}
 \psi(s)\sim \left\{
 \begin{array}{lll}
 1-c_1(\alpha)s^\alpha \quad &\mbox{for} \ s\to
0,\nonumber\\
\\
c_2(\alpha)s^{\frac{\alpha-1}{2}}e^{-s} \quad &\mbox{for} \ s\to
\infty,\nonumber
\end{array}
\right.
\end{equation}
where
$$c_1(\alpha)=\frac{2^{1-\alpha}\Gamma(1-\frac{1}{2}\alpha)}{\alpha\Gamma(\frac{1}{2}\alpha)} \quad \mbox{and}
\quad
c_2(\alpha)=\frac{2^{\frac{{1-\alpha}/2}{2}}\pi^{\frac{1}{2}}}{\Gamma(\frac{1}{2}\alpha)}.
$$
Moreover, $$-\lim\limits_{s\to 0^+}s^{1-\alpha}\psi'(s)=k_\alpha.$$ Recall that $k_\alpha=\alpha c_1(\alpha)$.
 Since $\lim\limits_{s\to 0^+}s^{1-\alpha}\psi'(s)<0$, there exists
some $\delta>0$ such that $\psi'(s)<0$ in $(0, \ \delta)$.
Let
$$\phi_\ast(X):=\phi(x)(\psi(\lambda^{\frac{1}{2}}y)-\psi(\lambda^{\frac{1}{2}}\delta)>0$$
and
$$\mathcal{C}_R:=\{(x, y)|(x_1-R)^2+|\tilde{x}|^2<1 \ \ \mbox{and} \ \ 0<y<\delta
\}$$ be the cylindrical domain. Here $\phi(x)$ is the first eigenfunction of the following eigenvalue problem
\begin{equation}
\left \{
\begin{array}{rll}
- \triangle \phi =& \lambda \phi  \quad \quad &\mbox{in} \ \{x| (x_1-R)^2+|\tilde{x}|^2<1\}, \medskip\\
\medskip
 \phi >&0 \quad \quad \quad &\mbox{in} \ \{x| (x_1-R)^2+|\tilde{x}|^2<1\}, \medskip\\
\phi=&0  \quad \quad &\mbox{on} \ \{x| (x_1-R)^2+|\tilde{x}|^2=1\}.
\end{array}
\right.
\end{equation}

It follows that
\begin{equation}
\left \{
\begin{array}{rll}
- L_\alpha \phi_\ast+ \frac{\lambda \phi_\ast \psi(\lambda^{\frac{1}{2}}\delta)}
{\psi(\lambda^{\frac{1}{2}}y)-\psi(\lambda^{\frac{1}{2}}\delta)}
=& 0 \quad \quad &\mbox{in} \ \mathcal{C}_R, \medskip\\
\medskip
 \phi_\ast= &0 \quad \quad &\mbox{on} \  \partial \mathcal{C}_R\backslash\{y=0\}, \medskip\\
\frac{\partial{\phi_\ast}}{\partial
\nu^{\alpha}}=&\frac{k_\alpha\lambda_1^{\frac{\alpha}{2}}\phi_\ast}{1-\psi(\lambda^{\frac{1}{2}}\delta)}\ \quad
\quad &\mbox{on} \  \partial \mathcal{C}_R\cap\{y=0\}.
\end{array}
\right.
\end{equation}
Since $w$ is nondecreasing in $x_1$ direction, from Proposition
\ref{pro1}, we obtain that
$$x_1w^{p-1}(X)\geq (R-1)m_0^{p-1}     $$
for $R>2$ and $X\in \mathcal{C}_R$, where $$m_0=\min_{X\in
\mathcal{C}_1}w(X) \quad \mbox{and} \quad \mathcal{C}_1=\{(x,
y)|(x_1-1)^2+|\tilde{x}|^2<1 \ \ \mbox{and} \ \ 0<y<\delta \}.$$ We can also see that
$w(X)$ satisfies the following

\begin{equation} \left \{
\begin{array}{rll}
- L_\alpha w =& 0 \quad \quad &\mbox{in} \ \mathcal{C}_R, \medskip\\
\medskip
w> &0 \quad \quad &\mbox{on} \  \partial \mathcal{C}_R\backslash\{y=0\}, \medskip \\
\frac{\partial{w}}{\partial \nu^{\alpha}}\geq &(R-1)m_0^{p-1}w
\quad \quad &\mbox{on} \  \partial \mathcal{C}_R\cap\{y=0\}.
\end{array}
\right.
\end{equation}
Set $$\psi_\ast:={\phi_\ast}/{w}>0.$$ Then $\psi_\ast(X)$ satisfies
\begin{equation} \left \{
\begin{array}{lll}
L_\alpha \psi_\ast+2\nabla \psi_\ast\cdot \frac{\nabla w}{w}-\frac{\lambda \psi_\ast \psi(\lambda^{\frac{1}{2}}\delta)}
{\psi(\lambda^{\frac{1}{2}}y)-\psi(\lambda^{\frac{1}{2}}\delta)} = 0 \quad \quad &\mbox{in} \ \mathcal{C}_R,
\medskip \\
\medskip
\psi_\ast= 0 \quad &\mbox{on} \  \partial \mathcal{C}_R\backslash\{y=0\}, \medskip \\
\frac{\partial{\psi_\ast}}{\partial \nu^{\alpha}}=
(\frac{k_\alpha\lambda_1^{\frac{\alpha}{2}}}{{1-\psi(\lambda^{\frac{1}{2}}\delta)}}-(R-1)m_0^{p-1})\psi_\ast
\quad &\mbox{on} \  \partial \mathcal{C}_R\cap\{y=0\}.
\end{array}
\right.
\end{equation}
If we choose $R$ sufficiently large, then
$$(\frac{k_\alpha\lambda_1^{\frac{\alpha}{2}}}{{1-\psi(\lambda^{\frac{1}{2}}\delta)}}
-(R-1)m_0^{p-1})\leq -C_0$$ for
some $C_0>0$. Thus
\begin{equation}
 \frac{\partial{\psi_\ast}}{\partial
\nu^{\alpha}}\leq -C_0\psi_\ast. \label{agu}
\end{equation}
By maximum principle, the maximum value value of $\psi_\ast$ should
be attained at some point on $\partial \mathcal{C}_R\cap\{y=0\}$.
Then $\frac{\partial{\psi_\ast}}{\partial \nu^{\alpha}}\geq 0$ at
that point. Obviously, it is a contradiction with (\ref{agu}). Hence
$\psi_\ast\equiv 0$ in $\mathcal{C}_R$. Then $\phi_\ast\equiv 0$,
which contradicts the construction of $\phi_\ast$. Therefore
$w\equiv 0$ in $\mathbb R^{n+1}_+$. We complete the proof of Theorem \ref{th1}.

\end{proof}

\section{ A priori estimates}
We apply the blow-up argument in \cite{GS} to obtain the a priori estimates. It reduces the a priori
bound to the results of Liouville theorems. We first recall two classical Liouville theorems for fractional
Laplacian in \cite{BCDS}.
\begin{lemma}
Let $1\leq \alpha<2$ and $1<p<\frac{n+\alpha}{n-\alpha}$. Then the
problem
\begin{equation}
\left \{
\begin{array}{rll}
- div(y^{1-\alpha}\nabla w) =& 0 \quad \quad &\mbox{in} \ \mathbb R^{n+1}_+,\medskip \\
\medskip
{w}>&0 \quad \quad &\mbox{in} \ \mathbb R^{n+1}_+,
\medskip \\
\frac{\partial {w}}{\partial \nu^{\alpha}}=& {w}^p \quad \quad
&\mbox{on} \
\partial\mathbb R^{n+1}_+
\end{array}
\right.
\end{equation}
has no bounded solution. \label{whol}
\end{lemma}

Let
$$\mathbb R^{n+1}_{++}:=\{ X=(x', x_n, y)| x_n>0, \ y>0\}.$$
\begin{lemma}
Let $1\leq \alpha<2$ and $1<p<\frac{n+\alpha}{n-\alpha}$. Then the
problem
\begin{equation}
\left \{
\begin{array}{rll}
- div(y^{1-\alpha}\nabla w) =& 0 \quad \quad &\mbox{in} \ \mathbb R^{n+1}_{++},\medskip\\
\medskip
\frac{\partial {w}}{\partial \nu^{\alpha}}=&{w^p(x', \ x_n, 0)} \quad &\mbox{on} \ \{X|y=0\}, \medskip\\
 w(x', 0, y)=& 0 \quad &\mbox{on} \ \{X|x_n=0\}
\end{array}
\right.
\end{equation}
has no positive bounded solution. \label{half}
\end{lemma}
Our Liouville theorem (i.e. Theorem \ref{th1}) is also essential in performing the blow-up argument
for (\ref{frac}). In order to get the a priori bound for (\ref{frac}), we shall consider Caffarelli-Silvestre's
extension, that is,
\begin{equation}
\left \{
\begin{array}{rll}
-div(y^{1-\alpha}\nabla w)=&0 \quad \quad &\mbox{in} \ C_{\Omega}, \medskip\\
w(x,y)=&0 \quad \quad &\mbox{on} \ \partial_LC_\Omega, \medskip \\
\frac{\partial w}{\partial \nu^\alpha}= &a(x)g(w) \quad \quad
&\mbox{on} \ \Omega\times \{0\}. \label{prio}
\end{array}
\right.
\end{equation}
Here $w(x,0)=u$ on $\Omega$, If one obtains the a priori bound for (\ref{prio}), then one proves
Theorem \ref{th2}. We shall prove the following proposition. The proof is an adaption of the argument
in \cite{GS} and \cite{BCN}.

\begin{proposition}
Assume that $1\leq \alpha<2$ and $1<p<\frac{n+\alpha}{n-\alpha}$.
Then there exists a generic constant $C$ such that  every solution
of (\ref{prio}) satisfies $$ \|w\|_\infty\leq C .$$
\end{proposition}
\begin{proof}

We prove it by contradiction. Suppose the conclusion in the proposition
is false. Then there exists a sequence of $\{w_j\}$ such that
$$ M_j=\|w_j\|_\infty\to \infty, \quad \mbox{as}\ j\to\infty.$$
By the maximum principle, there exists $(x^j, 0)$ such that
$M_j=w_j(x^j, 0)$. Let
$$ z=\frac{x-x^j}{\lambda_j}.$$
The positive scale factor $\lambda_j$ will be determined later with
$\lambda_j\to 0$ as $j\to\infty$. We introduce the rescaled function
$$V_j(z,\,y):=\frac{w_j(x^j+\lambda_jz, \lambda_j y)}{M_j}.$$
Let $$\Omega_j=\frac{1}{\lambda_j}(\Omega-x^j).$$ We can easily see
that
$$ \max V_j=V_j(0)=1. $$ A direct calculation shows that $V_j$
satisfies

\begin{equation}
\left \{
\begin{array}{rll}
-div(y^{1-\alpha}\nabla V_j)=&0 \quad \quad &\mbox{in} \ C_{\Omega_j},  \medskip \\
V_j=&0 \quad \quad &\mbox{on} \ \partial_LC_{\Omega_j},  \medskip \\
\frac{\partial V_j}{\partial \nu^\alpha}= &M_j^{-1}\lambda_j^\alpha
a(x^j+\lambda_j z)g(M_jV_j) \quad \quad &\mbox{on} \ \Omega_j\times
\{0\}.\label{limit}
\end{array}
\right.
\end{equation}
Since $x^j$ is bounded in $\Omega$, then $x^j\to x^0\in \bar \Omega$ as
$j\to\infty$. There are several cases for the location of the limit
point $x^0$.
Namely, \\
\indent\emph{Case 1:} $x^0\in \Omega^+\cap\Omega^-.$ \medskip

\indent\emph{Case 2:} $x^0\in \partial\Omega.$ \medskip

\indent\emph{Case 3:} $x^0 \in \Gamma$. \\
If Case 1 occurs, set
$$\lambda_j={M_j}^{\frac{1-p}{\alpha}}.$$
Let $d_j=dist\{x^j, \ \partial\Omega\}$. Since $x^0\in\Omega$, then
${d_j}/{\lambda_j}\to\infty$ as $j\to\infty$. The fact that $\mathbb
B_{{d_j}/{\lambda_j}}\subset \Omega_j$ implies that $\Omega_j\to \mathbb
R^n$ as $j\to\infty$. By regularity estimates as in the proof of Theorem 1, we have
\begin{equation}
\left \{
\begin{array}{lll}
V_j\rightharpoonup {V}\quad &\mbox{weakly in} \ H^1_{loc}(y^{1-\alpha}, \overline{\mathbb R^{n+1}_+}), \nonumber
\medskip \\
V_j\to {V} \quad & \mbox{in} \ C^{0,\ \beta}_{loc}(
\overline {R^{n+1}_+}) \nonumber
\end{array}
\right.
\end{equation}
for some $\beta>0$ and the equation (\ref{limit}) will turn into be a limit
equation
\begin{equation}
\left \{
\begin{array}{rll}
- div(y^{1-\alpha}\nabla V) =& 0 \quad \quad &\mbox{in} \ \mathbb R^{n+1}_+, \medskip\\
\medskip
{V}\geq &0 \quad \quad &\mbox{in} \ \mathbb R^{n+1}_+,
\medskip \\
\frac{\partial {V}}{\partial \nu^{\alpha}}=&a(x^0) {V}^p \quad \quad
&\mbox{on} \
\partial\mathbb R^{n+1}_+.
\end{array}
\right.
\end{equation}
In above, we have used the assumption (\ref{ggg}). The maximum principle implies that $a(x^0)>0$. We also have
$V(0)=1$. However, $V\equiv 0$ from Lemma \ref{whol}. A contradiction is arrived.

If  Case 2 occurs, then $d_j\to0$ as $j\to\infty$. We choose the
same $\lambda_j$ as case 1. We have two subcases for the ratio of
${d_j}/{\lambda_j}$, that is,\\
\indent\emph{Case
(a):} ${d_j}/{\lambda_j}\to \delta_0\geq 0$ for a subsequence. \medskip

\indent\emph{Case (b):} ${d_j}/{\lambda_j}\to \infty$  for a subsequence.

In case (a), after a limit procedure, the domain $\Omega_j$ converges
to (up to a rotation) some half space $H_{\delta_0}:=\{x\in \mathbb
R^n| x_n\geq -\delta_0\}$. We obtain that $V$ is a nonnegative solution of
\begin{equation}
\left \{
\begin{array}{rll}
- div(y^{1-\alpha}\nabla V) =& 0 \quad \quad &\mbox{in} \  H_{\delta_0}\times (0,\ \infty), \medskip\\
\medskip
{V}=&0 \quad \quad &\mbox{on} \ \partial H_{\delta_0}\times (0,\
\infty),
\medskip \\
\frac{\partial {V}}{\partial \nu^{\alpha}}=&a(x^0) {V}^p \quad \quad
&\mbox{on} \
 H_{\delta_0}\times \{0\}
\end{array}
\right.
\end{equation}
where $V(0)=1$. We can also see that $a(x^0)>0$. By a translation,
we can infer that  $V\equiv 0$ from Lemma \ref{half}. Clearly, it is
a contradiction. In case (b), if we carry out the same procedure as
case 1, we will also  arrive at a contradiction. We only need to
take care of Case 3.

If case 3 occurs, set
$$ \delta_j:= dist(x^j, \ \Gamma)=|x^j-z^j|, \quad z^j\in \Gamma.$$
Then $\delta_j\to 0$ as $j\to\infty$. Since $\nabla a\not =0$ on
$\Gamma$, it follows that $\delta_j$ is given by
\begin{equation}
\delta_j=\left\{ \begin{array}{lll} \frac{\nabla a(z^j)}{|\nabla
a(z^j)|}(x^j-z^j),\quad x^j\in \Omega^+ \nonumber, \medskip\\
 \nonumber\\
-\frac{\nabla a(z^j)}{|\nabla a(z^j)|}(x^j-z^j),\quad x^j\in
\Omega^-. \nonumber
\end{array}
\right.
\end{equation}
Since $a(z^j)=0$ and $a(x)\in C^2(\bar \Omega)$, by Taylor expansion, we
have
$$ a(x^j+\lambda_jz)=\pm |\nabla a(z^j)|\delta_j+\lambda_j\nabla a(z^j)\cdot z +O(\lambda_j^2| z|^2+\delta_j^2).$$
Substituting this identity into (\ref{limit}) yields that
\begin{equation}
\left \{
\begin{array}{lll}
-div(y^{1-\alpha}\nabla V_j)=0  \quad &\mbox{in} \ C_{\Omega_j}, \medskip \\
V_j=0  \quad &\mbox{on} \ \partial_LC_{\Omega_j}, \medskip \\
\vspace{1.5mm} \frac{\partial V_j}{\partial \nu^\alpha}=
M_j^{p-1}\lambda_j^\alpha (\pm |\nabla
a(z^j)|\delta_j+\lambda_j\nabla a(z^j)\cdot
z+O(\lambda_j^2|z|^2+\delta_j^2)){g(M_jV_j)}/{M_j^p}  \quad
&\mbox{on} \ \Omega_j\times \{0\}. \label{equ}
\end{array}
\right.
\end{equation}
 Observe that the third equation of (\ref{equ}) on $\Omega_j\times \{0\}$
holds in the ball
$$ |z|\leq \frac{1}{3\lambda_j}dist (x^0,\ \partial\Omega) $$
for large $j$. There are several subcases to consider.\\

\indent\emph{Case a:} ${\delta_j}/{M_j^{\frac{1-p}{1+\alpha}}}\to 0$
for a
subsequence. \\

We choose $\lambda_j=M_j^{\frac{1-p}{1+\alpha}}$. Note that
$\Omega_j$ tends to $\mathbb R^n$,
$$\delta_j\lambda_j^\alpha M_j^{p-1}\to 0 $$ and
$$O(\lambda_j^2|z|^2+\delta_j^2)\lambda_j^\alpha M_j^{p-1}\to 0 $$
for fixed $z$ as $j\to\infty$. By regularity estimates, $V_j\to V$
in $C^{0, \ \beta}_{loc}(\overline {R^{n+1}_+})$ and $V$ satisfies
\begin{equation}
\left \{
\begin{array}{rll}
- div(y^{1-\alpha}\nabla V) =& 0 \quad \quad &\mbox{in} \ \mathbb R^{n+1}_+,\medskip \\
\medskip
{V}\geq &0 \quad \quad &\mbox{in} \ \mathbb R^{n+1}_+,
\medskip \\
\frac{\partial {V}}{\partial \nu^{\alpha}}=&\nabla a(x^0)\cdot z
{V}^p \quad \quad &\mbox{on} \
\partial\mathbb R^{n+1}_+
\end{array}
\right.
\end{equation}
with $V(0)=1$. After a suitable rotation and rescaling, it becomes
\begin{equation}
\left \{
\begin{array}{rll}
- div(y^{1-\alpha}\nabla V) =& 0 \quad \quad &\mbox{in} \ \mathbb R^{n+1}_+,\medskip \\
\medskip
{V}\geq &0 \quad \quad &\mbox{in} \ \mathbb R^{n+1}_+,
\medskip \\
\frac{\partial {V}}{\partial \nu^{\alpha}}=& z_1 {V}^p \quad \quad
&\mbox{on} \
\partial\mathbb R^{n+1}_+.
\label{usef}
\end{array}
\right.
\end{equation}
Furthermore, $V(0)=1$. However, we know that the solution for
(\ref{usef}) is trivial from Theorem \ref{th1}. A contradiction is
arrived. \medskip\\
\indent\emph{Case b:} ${\delta_j}/{M_j^{\frac{1-p}{1+\alpha}}}\to
\infty$ for a
subsequence.
\medskip
\\
We select
$\lambda_j=\delta_j^{\frac{-1}{\alpha}}M_j^{\frac{1-p}{\alpha}}$.
Then
$$\lambda_j^{1+\alpha}M_j^{p-1}=\delta_j^{\frac{-(1+\alpha)}{\alpha}}M_j^{\frac{1-p}{\alpha}}\to
0$$ and
$$O(\lambda_j^2|z|^2+\delta_j^2)\lambda_j^\alpha M_j^{p-1}\to 0 $$
for fixed $z$ as $j\to\infty$. By regularity estimates, $V_j\to V$ and $V$
satisfies
\begin{equation}
\left \{
\begin{array}{rll}
- div(y^{1-\alpha}\nabla V) =& 0 \quad \quad &\mbox{in} \ \mathbb R^{n+1}_+,\medskip \\
\medskip
{V}\geq &0 \quad \quad &\mbox{in} \ \mathbb R^{n+1}_+,
\medskip \\
\frac{\partial {V}}{\partial \nu^{\alpha}}=&\pm |\nabla a(x^0)|
{V}^p \quad \quad &\mbox{on} \
\partial\mathbb R^{n+1}_+.
\end{array}
\right.
\end{equation}
Performing a rescaling, we know there exists only trivial solution, which
contradicts the fact the $V(0)=1$.

\indent\emph{Case c:} There exists some constant $\tilde \delta$ such
that ${\delta_j}/{M_j^{\frac{1-p}{1+\alpha}}}\to \tilde \delta$
for a
subsequence. \\

Let again $\lambda_j=M_j^{\frac{1-p}{1+\alpha}}$. Then  we have
$$O(\lambda_j^2|z|^2+\delta_j^2)\lambda_j^\alpha M_j^{p-1}\to 0 $$
for fixed $z$ as $j\to\infty$. By elliptic estimates, $V_j\to V$ and $V$ is the solution of
\begin{equation}
\left \{
\begin{array}{rll}
- div(y^{1-\alpha}\nabla V) =& 0 \quad \quad &\mbox{in} \ \mathbb R^{n+1}_+, \medskip\\
\medskip
{V}\geq &0 \quad \quad &\mbox{in} \ \mathbb R^{n+1}_+,
\medskip \\
\frac{\partial {V}}{\partial \nu^{\alpha}}=&(\pm |\nabla
a(x^0)|\tilde{\delta}+a(x^0)\cdot y){V}^p \quad \quad &\mbox{on} \
\partial\mathbb R^{n+1}_+.
\end{array}
\right.
\end{equation}
After a suitable rescaling, rotation and translation, it again
becomes
\begin{equation}
\left \{
\begin{array}{rll}
- div(y^{1-\alpha}\nabla V) =& 0 \quad \quad &\mbox{in} \ \mathbb R^{n+1}_+, \medskip\\
\medskip
{V}\geq &0 \quad \quad &\mbox{in} \ \mathbb R^{n+1}_+,
\medskip \\
\frac{\partial {V}}{\partial \nu^{\alpha}}=& z_1 {V}^p \quad \quad
&\mbox{on} \
\partial\mathbb R^{n+1}_+
\label{use}
\end{array}
\right.
\end{equation}
with $V(0)=1$. Clearly it is a contradiction from Theorem \ref{th1} again.
In conclusion, we obtain the a priori bound of solutions.

\end{proof}

\section{Proof of Theorem \ref{th3}}
In this section, we give the proof of Theorem \ref{th3}. We first
consider the nonexistence of solutions in the supercritical case,
i.e. $p>\frac{n+\alpha}{n-\alpha}$. For the subcritical and critical
cases, i.e. $1<p\leq \frac{n+\alpha}{n-\alpha}$, we consider the
solutions in a suitable higher dimension and reduce it to the
supercritical case. The idea is inspired by the work of \cite{LZ}.

\begin{proof}[Proof of Theorem \ref{th3}]

The proof of the theorem is divided in two cases. We shall show the nonexistence of solutions in both cases.\\

Case 1 (Supcritical case) : $p>\frac{n+\alpha}{n-\alpha}.$\\

Since no decay for the solution $w(X)$ is imposed at infinity, we
introduce the Kelvin transform, that is,
$$ \tilde w(X)=\frac{1}{|X|^{n-\alpha}}w(\frac{X}{|X|^2}).$$

Then $\tilde w$ satisfies
\begin{equation}
\left \{
\begin{array}{rll}
- div(y^{1-\alpha} \nabla \tilde w) =& 0 \quad \quad &\mbox{in} \
\mathbb
R^{n+1}_+, \medskip \\
\medskip
\frac{ \partial \tilde w}{\partial \nu^{\alpha} }=& -|X|^{\tilde
\beta} \tilde{w}^p(x,0) \quad \quad &\mbox{on} \
\partial\mathbb R^{n+1}_+\backslash\{0\},
\label{doo}
\end{array}
\right.
\end{equation}
where $\tilde \beta=p(n-\alpha)-(n+\alpha).$

Because of the Kelvin transform, the origin is the
singular point. We first prove a technical lemma to take care of the
origin.

\begin{lemma}
Assume that $\tilde w(X)$ satisfies (\ref{doo}). For all
$0<\epsilon<\min\{1, \ \frac{2\alpha}{\alpha+1}\min_{\partial\mathbb
B^+_1\cap\partial \mathbb B_1} \tilde w\}$, we have $\tilde w(X)\geq
\epsilon/2 $ for every $X\in \bar{ \mathbb B}^+_1 \backslash\{0\}$.
\label{tec2}
\end{lemma}
\begin{proof}
For $0<r<1$, we introduce the following test function
$$
\psi_1(X)=\frac{\epsilon}{2}-\frac{\alpha+1}{2\alpha}\frac{r^{n-\alpha}\epsilon}{|X|^{n-\alpha}}+\frac{\epsilon
y^\alpha}{2\alpha} \quad \forall x\in \mathbb B_1^+\backslash\mathbb
B_r^+ .$$ Set
$$A_1(X):=\tilde{w}(X)-\psi_1(X). $$
Direct calculation shows that
\begin{equation}
\left \{
\begin{array}{lll}
- div(y^{1-\alpha} \nabla  A_1) = 0 \quad \quad &\mbox{in} \
\mathbb B_1^+\backslash\mathbb
B_r^+, \medskip \\
\medskip
\frac{ \partial \tilde w}{\partial \nu^{\alpha} }= -|X|^{\tilde
\beta}w^p(x,0)+\frac{\epsilon}{2} \quad \quad &\mbox{on} \
\partial({\bar{ \mathbb B}_1^+\backslash \bar{\mathbb B}_r^+}) \cap \partial\mathbb
R^{n+1}_+.
\end{array}
\right.
\end{equation}
We claim that
\begin{equation} A_1(X)\geq 0\quad \ \ \mbox{in}\  \bar{
\mathbb B}_1^+\backslash \bar{\mathbb B}_r^+.
\label{clai}
\end{equation}
We show this claim by contradiction. On $\partial \mathbb B_r^+\cap \partial\mathbb B_r$, we have
$$
A_1(X)=\tilde{w}-(\frac{\epsilon}{2}-\frac{(\alpha+1)\epsilon}{2\alpha}+\frac{\epsilon
y^\alpha}{2\alpha})>\tilde{w}>0.$$ On $\partial \mathbb B_1^+\cap
\partial\mathbb B_1$, it follows that
$$A_1(X)=\tilde{w}-(\frac{\epsilon}{2}-\frac{(\alpha+1)\epsilon}{2\alpha}r^{n-\alpha}+\frac{\epsilon
y^\alpha}{2\alpha})>\tilde{w}-\frac{(\alpha+1)\epsilon}{2\alpha}>0.$$
If (\ref{clai}) is not true, by maximum principle, there exists some
$\bar X=(\bar x, 0)\in \partial \mathbb R^{n+1}_+$ with $r<|\bar
x|<1$ such that
$$ A_1(\bar X)=\min_{\bar{
\mathbb B}_1^+\backslash \bar{\mathbb B}_r^+}A_1(X)<0. $$ On one
hand,
$$\frac{\partial A_1}{\partial \nu^\alpha}=-\lim\limits_{y\to 0} y^{1-\alpha}\frac{ \partial A_1}{\partial y}
=-|\bar X|^{\tilde{\beta}}\tilde{w}^p(\bar X)+\frac{\epsilon}{2}\leq 0.$$
Thus
$$\tilde{w}^p(\bar X)\geq  \frac{\epsilon}{2}    $$
since $\tilde\beta>0$. On the other hand,
\begin{equation}
\begin{array}{lll}
A_1(\bar X)&=&\tilde{w}(\bar X)-(\frac{\epsilon}{2}-\frac{(\alpha+1)\epsilon}
{2\alpha}\frac{r^{n-\alpha}}{|\bar X|^{n-\alpha}}) \nonumber \medskip \\
&>& \tilde{w}(\bar X)-\frac{\epsilon}{2} \nonumber \medskip \\
&\geq & 0. \nonumber
\end{array}
\end{equation}
It contradicts that $A_1(\bar X)<0$. Hence we verify the claim. For $X\in \mathbb B_1 \backslash \{0\}$,  it follows that
$A_1(X)\geq 0$ for $0<r<|X|$. Let $r\to 0$, we have that $\tilde{ w}\geq \frac{\epsilon}{2}$.
\end{proof}

For $\lambda<0$,
let $$\Sigma_\lambda=\{ X| x_1\leq \lambda\}$$ and
$$ \tilde {\Sigma}_\lambda=\Sigma_\lambda\backslash \{0^\lambda\}.$$
Here $0^\lambda$ is the reflection point of $0$ with respect to the plane $T_\lambda$. Let
$$\tilde{v}_\lambda(X)=\tilde{w}_\lambda(X)-\tilde{w}(X).$$ Then $\tilde{v}_\lambda$ satisfies

\begin{equation}
\left \{
\begin{array}{lll}
- div(y^{1-\alpha} \nabla \tilde v_\lambda) = 0 \quad \quad &\mbox{in} \
\mathbb
R^{n+1}_+, \medskip \\
\medskip
\frac{ \partial \tilde v}{\partial \nu^{\alpha} }\geq  -p|X_\lambda|^{\tilde
\beta}\xi^{p-1}(X)\tilde v_\lambda(x,0) \quad \quad &\mbox{on} \
\partial\mathbb R^{n+1}_+\backslash\{0\},
\label{exis}
\end{array}
\right.
\end{equation}
where $\xi$ is a positive function between $\tilde w_\lambda$ and $ \tilde w$. We apply the moving plane method for
solutions of equation (\ref{exis}). Our goal  is to show that $\tilde w$ is symmetric with respect to $x_1=0$. The
proof consists of two steps.
\medskip \\
Step 1: If $\lambda$ is sufficiently negative, then $\tilde{v}_\lambda\geq 0$ for all $x\in \tilde{\Sigma}_\lambda$.
\medskip

Suppose it fails, then $\tilde{v}_\lambda<0$ somewhere in $\tilde{\Sigma}_\lambda$. Since $\tilde w_\lambda$ and
$\tilde w$ both converge to $0$ as $|X|\to \infty$, then $\tilde v_\lambda(X)\to 0$ as $|X|\to \infty$.
it follows from Lemma
\ref{tec2} that if $\lambda$ is sufficiently negative, $\tilde v_\lambda(X)>0$ for $X$ close to $0^\lambda$,
Thus there exists some point $\bar X$ such that
$$\tilde v_\lambda(\bar X)=\min_{X\in \tilde{\Sigma}_\lambda}\tilde v_\lambda(X)<0.$$ From the maximum principle, we
know that $\bar X\in \tilde{\Sigma}_\lambda\cap \partial \mathbb R^{n+1}_+$. Furthermore,
$\frac{\partial \tilde v_\lambda}{\partial \nu^\alpha}<0$, but it contradicts the second equation in (\ref{exis}).
Thus the plane can be moved to the right from the negative infinity. We assume that the plane will reach a critical
point. Define
$$\lambda_0=\sup\{\lambda<0 | \tilde{w}_\mu\geq 0 \ \ \mbox{in} \ \tilde{\Sigma}_\lambda \ \ \mbox{for all } \ \
 -\infty<\mu<\lambda\}.  $$
 \\
Step 2: We show that $\lambda_0=0$.
\medskip

If it is not true, it follows that $\lambda_0<0$. We claim that
\begin{equation}
 \tilde v_{\lambda_0}(X)\equiv 0,
 \end{equation}
which contradicts
$$\frac{ \partial \tilde v_\lambda}{\partial \nu^{\alpha} }=-|X_\lambda|^{\tilde
\beta}\tilde {w}^p_\lambda(X)+|X|^{\tilde
\beta}w^p(X)    \quad \mbox{for} \ X\in \partial \mathbb R^{n+1}_+.$$ We also show it by contradiction. By the
maximum principle,
$\tilde v_{\lambda_0}(X)>0 $ in $\tilde{\Sigma}_{\lambda_0}\backslash T_{\lambda_0}$. we need the following lemma
to take care of the singular point.
\begin{lemma}
For $r_0<\min\{\frac{1}{2}\lambda_0, \ 1\}$, there exists some positive constant $c$ such that
$\tilde v_{\lambda_0}(X)>c $ in $\mathbb B^+_{r_0}(0^{\lambda_0})\backslash\{0^{\lambda_0}\}$.
\label{tec3}
\end{lemma}
\begin{proof}
Since $ \tilde{v}_{\lambda_0}>0$ in $\tilde \Sigma_{\lambda_0}\cap \mathbb B^+_{r_0}(0^{\lambda_0})$, then
$\min_{\partial \mathbb B^+_{r_0}(0^{\lambda_0})}\tilde{v}_{\lambda_0}\geq \epsilon $ for some $0<\epsilon<1$.
By the continuity of $\tilde w$ in $\mathbb R^{n+1}_+\backslash\{0\}$, there exists some positive constant $c_1$ such that
\begin{equation}
 \tilde w(X)<c_1 \quad \mbox{for} \ X\in \bar{\mathbb B}^+_{r_0}(0^{\lambda_0}).
 \label{dep}
 \end{equation}
Let
$$\psi_2(X)=\frac{\epsilon \mu}{2}-\frac{r^{n-\alpha}\epsilon}{|X|^{n-\alpha}}+\frac{\epsilon (1-\mu)y^\alpha}{2}
\quad \mbox{in} \ \mathbb B^+_{r_0}(0^{\lambda_0})\backslash  \mathbb B^+_{r}(0^{\lambda_0}), $$
where the positive constant $\mu$ will be determined later. Set
$$ A_2(X)=\tilde{v}_{\lambda_0}(X)-\psi_2(X).$$
Direct calculation shows that
\begin{equation}
\left \{
\begin{array}{lll}
- div(y^{1-\alpha} \nabla  A_2) = 0 \quad \quad &\mbox{in} \
\mathbb B^+_{r_0}(0^{\lambda_0})\backslash  \mathbb B^+_{r}(0^{\lambda_0}), \medskip\\
\frac{ \partial A_2}{\partial \nu^{\alpha} }= -|X^{\lambda_0}|^{\tilde
\beta}\tilde w^p_{\lambda_0}(x,0)+|X|^{\tilde
\beta}\tilde w^p(x,0)+\frac{\epsilon(1-\mu)}{2} \quad \quad &\mbox{on} \
\partial({\bar {\mathbb B}^+_{r_0}(0^{\lambda_0})\backslash  \bar{\mathbb B}^+_{r}(0^{\lambda_0})}) \cap \partial\mathbb
R^{n+1}_+.
\label{mmm}
\end{array}
\right.
\end{equation}
We also claim that
$$ A_2(X)\geq 0  \quad \mbox{in} \ \mathbb B^+_{r_0}(0^{\lambda_0})\backslash  \mathbb B^+_{r}(0^{\lambda_0}). $$
On $\partial \mathbb B^+_{r_0}(0^{\lambda_0})\cap \partial\mathbb B_{r_0}(0^{\lambda_0})$,
$$ A_2(X)\geq \epsilon-(\frac{\epsilon \mu}{2}-\frac{r^{n-\alpha}
\epsilon}{|r_0|^{n-\alpha}}+\frac{\epsilon(1-\mu)}{2})>0.$$
On $\partial \mathbb B^+_{r}(0^{\lambda_0})\cap \partial\mathbb B_{r}(0^{\lambda_0})$,
$$ A_2(X)>\tilde w_{\lambda_0}(X)>0. $$
By the maximum principle, the minimum value of $A_2(X)$ shall occur on
$\partial({\bar {\mathbb B}^+_{r_0}(0^{\lambda_0})\backslash  \bar{\mathbb B}^+_{r}(0^{\lambda_0})}) \cap \partial\mathbb
R^{n+1}_+$. So there exists some $\bar X\in \partial({\bar {\mathbb B}^+_{r_0}(0^{\lambda_0})\backslash  \bar{\mathbb B}^+_{r}(0^{\lambda_0})}) \cap \partial\mathbb
R^{n+1}_+ $ such that
$$ A_2(\bar X)<0 \quad \mbox{and} \ \ \frac{\partial A_2}{\partial \nu^\alpha}(\bar X)\leq 0.$$
The fact that $ A_2(\bar X)<0$ implies that
$$\tilde w_{\lambda_0}(\bar X)-\tilde w(\bar X)-\psi_2(\bar X)<0.$$
Then
\begin{equation}
\tilde w_{\lambda_0}(\bar X)\leq c_2,
\label{dep1}
\end{equation}
 where $c_2$ only depends on $c_1$. Furthermore, $A_2(\bar X)<0$ implies that
 \begin{equation}
\tilde{v}_{\lambda_0}(\bar X)\leq \frac{\epsilon \mu}{2}-\frac{r^{n-\alpha}
\epsilon}{|\bar X|^{n-\alpha}}<\frac{\epsilon \mu}{2}.
\label{nee}
 \end{equation}
With the help of (\ref{dep}),(\ref{dep1}) and Mean value theorem, we obtain
\begin{equation}
\begin{array}{lll}
|\bar X^{\lambda_0}|^{\tilde
\beta}\tilde w^p_{\lambda_0}(\bar x,0)-|\bar X|^{\tilde
\beta}\tilde w^p(\bar x,0)&\leq& |\bar X^{\lambda_0}|^{\tilde
\beta}(\tilde w^p_{\lambda_0}(\bar X)-\tilde w^p(\bar X)) \nonumber \medskip \\
&\leq & c_3 \tilde v_{\lambda_0}(\bar X) \nonumber\\
\end{array}
\end{equation}
where $c_3$ depends on $\lambda_0$, $c_1$ and $c_2$. Since $\frac{\partial A_2}{\partial \nu^\alpha}\leq 0, $
by the second equation of (\ref{mmm}), we infer that
\begin{equation}
\frac{\epsilon (1-\mu)}{2 c_3}\leq \tilde{v}_{\lambda_0}(\bar X).
\label{nee3}
\end{equation}
Together with (\ref{nee}) and (\ref{nee3}), we have
$$ \frac{\epsilon (1-\mu)}{2 c_3}<\frac{\epsilon \mu}{2}.      $$
If we choose $\mu$ small enough such that $\mu<{1}/{(1+c_3)}$ at the beginning,
 we will reach a contradiction. Hence we prove the claim.
Let $r\to 0$. Hence
$$\tilde{v}_{\lambda_0}(X)>c=\frac{\epsilon \mu}{2} $$
for $\mu<{1}/{(1+c_3)}$. This completes the proof.
\end{proof}

We continue the proof of Step 2. By the definition of $\lambda_0$, there exist sequences of $\lambda_k (
\lambda_k>\lambda_0)$ and $\bar X_k$ such that $\lambda_k\to\lambda_0$ and $\tilde v_{\lambda_k}(\bar X_k)
=\inf_{\tilde \Sigma_{\lambda_k}}\tilde v_{\lambda_k}(X)<0$. By Lemma \ref{tec3} and continuity of $
\tilde v_{\lambda_k}$, we have
$$ \tilde v_{\lambda_k}(X)>\frac{c}{2} \quad \forall X\in  \mathbb B^+_{r_0}(0^{\lambda_0})\backslash\{0^{\lambda_0}\} $$
for $k$ large enough, since $\tilde v_{\lambda_k}(X)\to 0$ as $|X|\to \infty$. By the maximum principle,
$\bar X_k\in (\tilde \Sigma_{\lambda_k}\backslash \mathbb B^+_{r_0}(0^{\lambda_0}))\cap \partial \mathbb R^{n+1}_+$. The same
argument as Step 1 gives the contradiction. Therefore, it is confirmed that $\lambda_0=0$.

It is clear that $w(x, y)$ is symmetric with respect to $x_1=0$. Since the equation is invariant under rotation,
we conclude that $w(x,y)$ is radially symmetric with respect to the point $(0, y)$ for every fixed $y$.
Thanks to the Kelvin transform, we can choose the origin arbitrarily on the plane $y=0$. Thus, $w(x, y)$ only
depends on the variable $y$. The degenerate partial differential equation (\ref{new}) turns out to be ordinary differential equation, that is,
\begin{equation}
\left\{
\begin{array}{lll}
w_{yy}+\frac{1-\alpha}{y}w_y&=&0 \quad \forall \ y>0, \nonumber \medskip \\
-\lim\limits_{y\to 0^+}y^{1-\alpha} w_y&=&w^p(0).
 \nonumber
\end{array}
\right.
\end{equation}
Solving the ordinary differential equation gives that
$$ w=\frac{a}{\alpha}y^\alpha+b$$
where $b>0$ and $ a=b^p$.
\medskip \\
Case 2 (Subcritical and critical cases) : $1<p\leq \frac{n+\alpha}{n-\alpha}.$\\

We expand the dimension of the space and reduce the problem into supercritical case. We
choose a large integer $m$ such that
$$p>\frac{n+m+\alpha}{n+m-\alpha}.$$
Let
$$ w^0(x_1,\cdots, x_n, x_{n+1},\cdots, x_{n+m}, y)=w(x_1,\cdots, x_n, y).$$
Then $w^0$ satisfies

\begin{equation}
\left \{
\begin{array}{rll}
- div(y^{1-\alpha} \nabla w^0) =& 0 \quad \quad &\mbox{in} \ \mathbb
R^{n+m+1}_+, \medskip\\
\medskip
\lim\limits_{y\to 0^+} y^{1-\alpha}\frac{ \partial w^0}{\partial y}(x,
y)=& (w^0)^p \quad \quad &\mbox{on} \
\partial\mathbb R^{n+m+1}_+.
\label{expa}
\end{array}
\right.
\end{equation}
Observe that $p$ is supercritical in the equation (\ref{expa}). Applying the
same argument as Case 1, we deduce that $w^0$ is independent of $x_i$ for $i=1, \cdots, n+m$. Thus,
$w$ is independent of $x_1, \cdots, x_n$. Again $w$ only depends on $y$. Therefore, $w$ satisfies
the same conclusion as Case 1. This completes the proof of Theorem \ref{th3}.

\end{proof}

\end{document}